\documentclass{article}

\usepackage{graphicx}
\usepackage{amsmath}
\usepackage{amssymb}
\usepackage{amsthm}
\usepackage{mathrsfs}
\usepackage{psfrag}
\usepackage{booktabs}
\usepackage{natbib}

\newcommand{\field}[1]{\mathbb{#1}}

\newcommand{\fs}[1]{\mathcal{#1}}

\newcommand{\nats}{\field{N}}

\newcommand{\define}{:=}
\newcommand{\comment}[1]{}

\newcommand{\agent}{\mathbf{P}}     
\newcommand{\defterm}[1]{\textbf{#1}}
\newcommand{\truth}{\mathbf{T}}     
\newcommand{\belief}{\mathbf{B}}    
\newcommand{\cause}{\mathbf{C}}     

\theoremstyle{plain}
\newtheorem{theorem}{Theorem}

\newtheorem{remark}{Remark}
\theoremstyle{definition}
\newtheorem{definition}{Definition}

\begin{document}

\title{Logic, Reasoning under Uncertainty \\ and Causality}

\author{Pedro A. Ortega}

\maketitle


\begin{abstract}
A simple framework for reasoning under uncertainty and intervention is
introduced. This is achieved in three steps. First, logic is restated in
set-theoretic terms to obtain a framework for reasoning under certainty.
Second, this framework is extended to model reasoning under uncertainty.
Finally, causal spaces are introduced and shown how they provide enough
information to model knowledge containing causal information about the world.
\end{abstract}

\section{Bayesian Probability Theory}

It is advantageous to endow plausibilities with an explanatory framework that
has a logically intuitive appeal. Such a framework is Bayesian probability
theory. Simply put, Bayesian probability theory is a framework that extends
logic for reasoning under uncertainty.

\subsection{Reasoning under Certainty}

Logic is the most important framework of reasoning (under certainty). Here, it
is rephrased in set-theoretic terms\footnote{Strictly speaking, this
set-theoretic logic is ``a logic within logic'', since set theory is based on
standard logic.}. As will be seen, this facilitates its extension to a
framework for reasoning under uncertainty.

Let $\Omega$ be a set of
\defterm{outcomes}, which is assumed to be finite for simplicity. A subset
$A \subset \Omega$ is an \defterm{event}. Let ${}^c$, $\cup$ and $\cap$ be the
set-operations of \defterm{complement}, \defterm{union} and
\defterm{intersection} respectively. Let $\fs{F}$ be an
\defterm{algebra}, i.e.\ a set of events obeying the axioms
\begin{itemize}
  \item[A1.] $\fs{F} \neq \varnothing$.
  \item[A2.] $A \in \fs{F} \quad \Rightarrow \quad A^c \in \fs{F}$.
  \item[A3.] $A, B \in \fs{F} \quad \Rightarrow
    \quad A \cup B \in \fs{F}$.
\end{itemize}
In this framework, an outcome $\omega \in \Omega$ is a state of affairs and an
event $A \in \fs{F}$ is a proposition. Hence, a singleton $\{\omega\} \in
\fs{F}$ is an irreducible (i.e.\ atomic) proposition about the world. The
set-operations ${}^c$, $\cup$ and $\cap$ correspond to the logical connectives
of $\neg$ (negation), $\vee$ (disjunction) and $\wedge$ (conjunction)
respectively. They allow the construction of complex propositions from simpler
ones. An algebra is a system of propositions that is closed under negation and
disjunction (and hence is closed under conjunction as well), i.e.\ it comprises
all propositions that the reasoner might entertain.

\begin{remark}
A trivial consequence of the axioms is that both the universal event $\Omega$
and the impossible event $\varnothing$ are in $\fs{F}$.
\end{remark}

The objective of logic is to allow the reasoner to conclude the veracity of
events given information. Let $\fs{V} \define \{1, 0, ?\}$ be the set of
\defterm{truth states}, where 1 is \defterm{true}, 0 is \defterm{false}, and
{?} is \defterm{uncertain} (but known to be either true or false). From these,
$\{1, 0\}$ are called \defterm{truth values}. The \defterm{truth function} is
the set function $\truth$ over $\fs{F} \times \fs{F}$ defined as
\[
    A, B \in \fs{F}, \quad \truth(A|B)
        = \begin{cases}
            1   & \text{if $B \subset A$}, \\
            0   & \text{if $A \cap B = \varnothing$}, \\
            {?} & \text{else.}
            \end{cases}
\]
Furthermore, define the shorthand $\truth(A) \define \truth(A|\Omega)$. The
quantity $\truth(A|B)$ stands for the ``truth value of event~$A$ given that
event~$B$ is true''. Accordingly, the knowledge of the reasoner about the facts
of the world is represented by his truth function and his algebra. From his
point of view, a proposition can be either true, false or uncertain (i.e.\
having an unresolved truth value given his knowledge). Understanding the
definition of the truth function is straightforward. Claiming that an event~$B
\in \fs{F}$ \emph{is true} means that one of its members $\omega \in B$ is the
current outcome/state of affairs. Hence the veracity of $A$ given $B$ is
evaluated as follows (Figure~\ref{fig:truth-space}): if~$A$ contains every
outcome in~$B$ then it must be true as well; if~$A$ is known not to contain any
of $B$'s outcome then it must be false; and if~$A$ contains only part of $B$
then it cannot be resolved, since knowing that $\omega \in B$ does neither
imply that $\omega \in A$ nor $\omega \in A^c$. The definition of a truth space
follows.

\begin{figure}[htbp]
\begin{center}
    \small
    \psfrag{p1}[c]{(a)}
    \psfrag{p2}[c]{(b)}
    \psfrag{p3}[c]{(c)}
    \psfrag{p4}[c]{(d)}
    \psfrag{e1}[c]{$\truth(A|A) = 1$}
    \psfrag{e2}[c]{$\truth(A^c|A) = 0$}
    \psfrag{e3}[c]{$\truth(C|A) = 1$}
    \psfrag{e4}[c]{$\truth(B|A) = {?}$}
    \psfrag{l1}[c]{$A$}
    \psfrag{l2}[c]{$B$}
    \psfrag{l3}[c]{$C$}
    \psfrag{l4}[c]{$\Omega$}
    \includegraphics[]{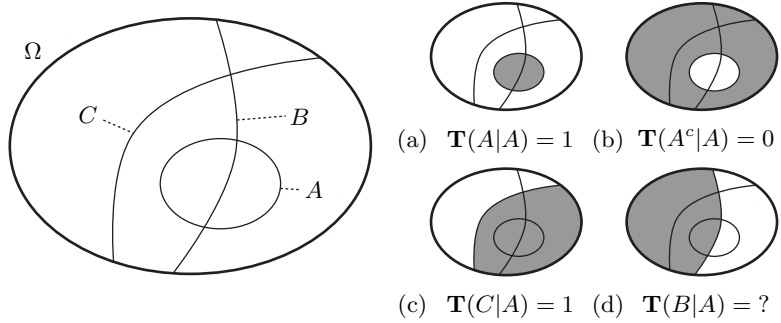}
    \caption[Truth Space.]{A truth space. It is known that the true outcome
    $\omega \in \Omega$ is in~$A$. Hence, (a) the event~$A$ is true
    and (b) its complement~$A^c$ is false. (c) Any event that contains
    a true event is true as well. (d) An event that contains only
    part of a true event is uncertain.}
    \label{fig:truth-space}
\end{center}
\end{figure}

\begin{definition}[Truth Space]
A \defterm{truth space} is a tuple $(\Omega, \fs{F}, \truth)$ where: $\Omega$
is a set of outcomes, $\fs{F}$ is an algebra over $\Omega$ and $\truth: \fs{F}
\times \fs{F} \rightarrow \fs{V}$ is a truth function.
\end{definition}

The intuitive meaning of a truth space is as follows. Nature arbitrarily
selects an outcome $\omega \in \Omega$. (This choice is \emph{not} governed by
a generative law.) \emph{Subsequently}, the reasoner performs a measurement: he
chooses a set~$B$ and nature reveals to him whether $\omega \in B$ or not.
Accordingly, the reasoners infers the veracity of any event~$A \in \fs{F}$ by
evaluating either $\truth(A|B)$ (if $\omega \in B$) or $\truth(A|B^c)$ (if
$\omega \notin B$).

Several measurements are combined as a conjunction. Thus, if the reasoner
learns that $\omega$ is in $B_1$, $B_2$, \ldots, and $B_t$ after performing~$t$
measurements, then the truth value is $\truth(A|B_1 \cap \cdots \cap B_t)$ for
any $A \in \fs{F}$.

\begin{remark}
Knowing that $\omega \in \Omega$ does not resolve uncertainty, i.e.\
$\truth(A|\Omega) = {?}$ for any $A \in \fs{F} \setminus \{\Omega,
\varnothing\}$, while knowing that $\omega \in \{\omega\}$ resolves all
uncertainty, i.e.\ $\truth(A|\{\omega\}) \in \{0,1\}$ for any $A \in \fs{F}$.
\end{remark}

\begin{remark}
The set relation $B \subset A$ corresponds to the logical relation $B
\Rightarrow A$. Since an algebra is an encoding of how sets are contained
within each other, it should be clear that an algebra is essentially a system
of implications.
\end{remark}

\subsection{Reasoning under Uncertainty}\label{sec:reasoning-under-uncertainty}

Unlike logic, Bayesian probability theory allows reasoning under uncertainty.
For this end, it provides a consistent mechanism to replace the uncertainty
state~${?}$ with a numerical value in the interval $[0,1]$ representing degrees
of truth, belief or plausibility.

\begin{figure}[htbp]
\begin{center}
    \small
    \psfrag{l1}[c]{(a)}
    \psfrag{l2}[c]{(b)}
    \psfrag{l3}[c]{(c)}
    \psfrag{s1}[c]{$A$}
    \psfrag{s2}[c]{$B$}
    \psfrag{s3}[c]{$C$}
    \psfrag{s4}[c]{$\Omega$}
    \psfrag{r2}[c]{$D$}
    \psfrag{p1}[c]{$\belief(A|B) = \truth(A|B) = 0$}
    \psfrag{p2}[c]{$\belief(A|B) = \frac{\belief(A \cap B|\Omega)}{\belief(B|\Omega)}$}
    \psfrag{p3}[c]{$\belief(A|B) = \truth(A|B) = 1$}
    \includegraphics[width=\textwidth]{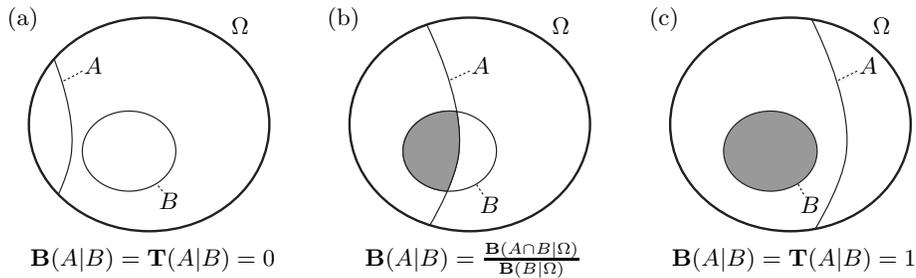}
    \caption{Extension of Truth Function.}
    \label{fig:probability-extension}
\end{center}
\end{figure}

The goal is to find a suitable definition of a quantity $\belief(A|B)$ meaning
``the degree of belief in event~$A$ given that event~$B$ is true'' that is
consistent with the truth function when it is certain, i.e. $\belief(A|B)
\define \truth(A|B)$ if $\truth(A|B) \in \{0,1\}$. Consider the three situations
in Figure~\ref{fig:probability-extension}. (a) In the case $A \cap B =
\varnothing$, we impose $\belief(A|B) \define \truth(A|B) = 0$. (b) In the case
$B \subset A$, we impose $\belief(A|B) \define \truth(A|B) = 1$. (c) In the
intermediate case where $\truth(A|B) = {?}$, the event~$A$ only partially
covers the members of~$B$. If one interprets the quantity $\belief(C|D)$ as
``the fraction of~$D$ contained in~$C$'', then one can characterize
$\belief(A|B)$ with the relation
\[
    \belief(A|B) = \frac{ \belief(A \cap B|\Omega) }{ \belief(B|\Omega) }
\]
as long as $\belief(B|\Omega) > 0$. It is easy to see that this formula
generalizes correctly to the border cases, since $\belief(A|B) =
\frac{0}{\belief(B|\Omega)} = 0$ when $A \cap B = \varnothing$ and
$\belief(A|B) = \frac{\belief(B|\Omega)}{\belief(B|\Omega)} = 1$ when $B
\subset A$. Noting that $B = B \cap \Omega$ and rearranging terms, one gets
\[
    \belief(A \cap B|\Omega) = \belief(B|\Omega) \, \belief(A|B \cap \Omega).
\]
This relation should hold under any restriction to a ``universal'' set $C \in
\fs{F}$, not only when it is restricted to~$\Omega$. Thus, replacing~$\Omega$
by~$C$ one obtains
\[
    \belief(A \cap B|C) = \belief(B|C) \, \belief(A|B \cap C),
\]
which is known as the \defterm{product rule} for beliefs. Following a similar
reasoning, we impose that for any event~$A \in \fs{F}$, the sum of the degree
of belief in~$A$ and its complement~$A^c$ must be true under any condition~$B$,
i.e.\
\[
    \belief(A|B) + \belief(A^c|B) = 1,
\]
which is known as the \defterm{sum rule} for beliefs. In summary, we impose the
following axioms for beliefs.

\begin{definition}[Belief axioms]\label{def:belief-axioms}
Let $\Omega$ be a set of outcomes and let $\fs{F}$ be an algebra over $\Omega$.
A set function $\agent$ over $\fs{F} \times \fs{F}$ is a \defterm{belief
function} iff
\begin{itemize}
    \item[B1.] $A, B \in \fs{F},
        \quad \belief(A|B) \in [0,1]$.
    \item[B2.] $A, B \in \fs{F},
        \quad \belief(A|B) = 1 \quad$ if $B \subset A$.
    \item[B3.] $A, B \in \fs{F},
        \quad \belief(A|B) = 0 \quad$ if $A \cap B = \varnothing$.
    \item[B4.] $A, B \in \fs{F},
        \quad \belief(A|B) + \belief(A^c|B) = 1$.
    \item[B5.] $A, B, C \in \fs{F},
        \quad \belief(A \cap B|C) = \belief(A|C) \, \belief(B|A \cap C)$.
\end{itemize}
\end{definition}
Furthermore, define the shorthand $\belief(A) \define \belief(A|\Omega)$.
Axiom~B1 states that degrees of belief are real values in the unit interval
$[0,1]$. Axioms~B2 and~B3 equate the belief and the truth function under
certainty. Axioms~B4 and~B5 are the structural requirements under uncertainty
discussed above. Accordingly, one defines a belief space as follows.

\begin{definition}[Belief Space]
A \defterm{belief space} is a tuple $(\Omega, \fs{F}, \belief)$ where: $\Omega$
is a set of outcomes, $\fs{F}$ is an algebra over $\Omega$ and $\belief: \fs{F}
\times \fs{F} \rightarrow [0,1]$ is a belief function.
\end{definition}

The intuitive meaning of a belief space is analogous to a truth space. Nature
arbitrarily selects an outcome $\omega \in \Omega$. \emph{Subsequently}, the
reasoner performs a measurement: he chooses a set~$B$ and nature reveals to him
whether $\omega \in B$ or not. Accordingly, the reasoners infers the degree of
belief in any event~$A \in \fs{F}$ by evaluating either $\belief(A|B)$ (if
$\omega \in B$) or $\belief(A|B^c)$ (if $\omega \notin B$).

\begin{remark}
The word ``subsequently'', that has been emphasized for the second time now, is
crucial. When the reasoner performs his measurements, the outcome is already
determined.
\end{remark}

\emph{An easy but fundamental result is that the axioms of belief are
equivalent to the axioms of probability\footnote{More precisely, the axioms of
beliefs as stated here imply the axioms of probability for finitely additive
measures over finite algebras. Furthermore, the axioms of beliefs also specify
a unique version of the conditional probability measure.}.} This simple
observation is what constitutes the foundation of Bayesian probability theory.

\subsection{Bayes' Rule}

We now return to the central topic of this chapter. Suppose the reasoner has
uncertainty over a set of competing hypotheses about the world. Subsequently,
he makes an observation. He can use this observation to update his beliefs
about the hypotheses. The following theorem explains how to carry out this
update.

\begin{theorem}[Bayes' Rule]
Let $(\Omega, \fs{F}, \belief)$ be a belief space. Let $\{ H_1, \ldots, H_N \}$
be a partition of $\Omega$, and let $D \in \fs{F}$ be an event such that
$\belief(D) > 0$. Then, for all $n \in \{1, \ldots, N\}$,
\[
    \belief(H_n|D)
    =
    \frac{ \belief(D|H_n) \, \belief(H_n) }
         { \belief(D) }
    =
    \frac{ \belief(D|H_n) \, \belief(H_n) }
         { \sum_m \belief(D|H_m) \, \belief(H_m) }.
\]
\end{theorem}

The interpretation is as follows. The $H_1, \ldots, H_N$ represent $N$ mutually
exclusive \defterm{hypotheses}, and the event~$D$ represents an new observation
or \defterm{data}. Initially, the reasoner holds a \defterm{prior belief}
$\belief(H_n)$ over each hypothesis $H_n$. Subsequently, he incorporates the
observation of the event $D$ and arrives at a \defterm{posterior belief}
$\belief(H_n|D)$ over each hypothesis $H_n$. Bayes' rule states that this
update can be seen as combining the prior belief $\belief(H_n)$ with the
\defterm{likelihood} $\belief(D|H_n)$ of observation $D$ under hypothesis $H_n$.
The denominator $\sum_m \belief(D|H_m) \belief(H_m) = \belief(D)$ just plays
the r\^{o}le of a normalizing constant (Figure~\ref{fig:bayes-rule}).

\begin{figure}[htbp]
\begin{center}
    \small
    \psfrag{l1}[c]{$\Omega$}
    \psfrag{h1}[c]{$H_1$}
    \psfrag{h2}[c]{$H_2$}
    \psfrag{h3}[c]{$H_3$}
    \psfrag{d}[c]{$D$}
     \includegraphics[]{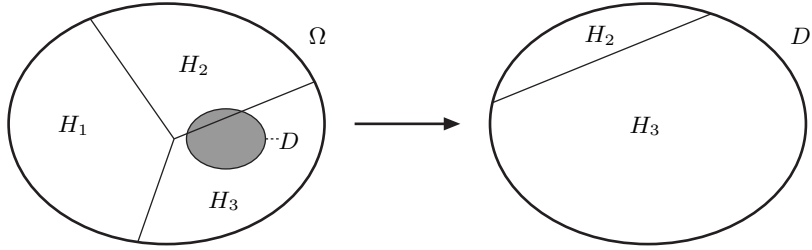}
    \caption[Bayes' rule.]{Schematic Representation of Bayes' Rule. The prior belief
    in hypotheses $H_1$, $H_2$ and $H_3$ is roughly uniform. After conditioning on the
    observation~$D$, the belief in hypothesis~$H_3$ increases significantly.}
    \label{fig:bayes-rule}
\end{center}
\end{figure}

Bayes' rule naturally applies to a sequential setting. Incorporating a new
observation~$D_t$ after having observed $D_1, D_2, \ldots, D_{t-1}$ updates the
beliefs as
\[
    \belief(H_n|D_1 \cap \cdots \cap D_t) =
    \frac{ \belief_t(D_t|H_n) \, \belief_t(H_n) }
         { \sum_m \belief_t(D_t|H_m) \, \belief(H_m) },
\]
where for the $t$-th update,
\[
    \belief_t(H_n) \define \belief(H_n|D_1 \cap \cdots \cap D_{t-1})
    \quad \text{and} \quad
    \belief_t(D_t|H_n) \define \belief(D_t|H_n \cap D_1 \cap \cdots \cap D_{t-1})
\]
play the r\^{o}le of the prior belief and the likelihood respectively. Note
that
\[
    \belief(D_1 \cap \cdots \cap D_t|H_n)
    = \prod_{\tau=1}^t \belief(D_\tau|H_n \cap D_1 \cap \cdots \cap D_{\tau-1}),
\]
and hence each hypothesis $H_n$ naturally determines a probability measure
$\belief(\cdot|H_n)$ over sequences of observations.

\begin{figure}[htbp]
\begin{center}
    \small
    \psfrag{l1}[c]{$\Omega$}
    \psfrag{h1}[c]{$H_1$}
    \psfrag{h2}[c]{$H_2$}
    \psfrag{h3}[c]{$H_3$}
    \psfrag{d1}[c]{$S_1$}
    \psfrag{d2}[c]{$S_2$}
    \psfrag{d3}[c]{$S_3$}
    \psfrag{d4}[c]{$S_4$}
    \psfrag{d5}[c]{$S_5$}
    \includegraphics[]{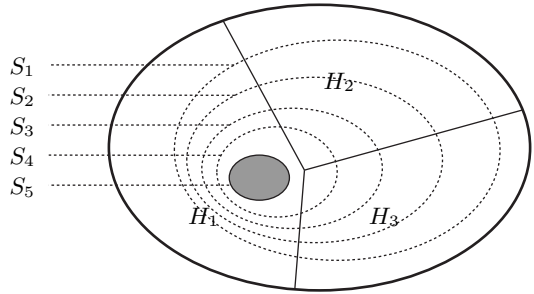}
    \caption[Progressive Refinement of Accuracy.]{Progressive refinement of
    the accuracy of the joint observation. The sequence of observations
    $D_1, \ldots, D_5$ leads to refinements $S_1, S_2, \ldots, S_5$, where
    $S_t = D_1 \cap \cdots \cap D_t$. Note that $S_5 \subset H_1$
    and therefore $\belief(H_1|S_5) = 1$, while $\belief(H_2|S_5) = \belief(H_3|S_5) = 0$.}
    \label{fig:data-sequence}
\end{center}
\end{figure}

A smaller event~$D$ corresponds to a more ``accurate'' observation. Hence,
making a new observation~$D'$ necessarily improves the accuracy, since
\[
    D \supset D \cap D'.
\]
In some cases, the accuracy of an observation (or sequence of observations) can
be so high that it uniquely identifies a hypothesis
(Figure~\ref{fig:data-sequence}).

The way Bayes' rule operates can be illustrated as follows. Consider a
partition $\{X_1, \ldots, X_K\}$ of $\Omega$ and let $H_\ast \in \{H_1, \ldots,
H_N\}$ be the true hypothesis, i.e.\ the outcome $\omega \in \Omega$ is drawn
obeying propensities described by $\belief(\cdot|H_\ast)$. The $X_k$ represent
different observations the reasoner can make. If $\omega$ is drawn and reported
to be in $X_k$, then the log-posterior probability of hypothesis $H_n$ is given
by
\[
    \log \belief(H_n|X_k)
    = \underbrace{\log \belief(X_k|H_n)}_{l_n}
    + \underbrace{\log \belief(H_n)}_{p_n}
    - \underbrace{\log \belief(X_k)}_{c}.
\]
This decomposition highlights all the relevant terms for understanding Bayesian
learning. The term~$l_n$ is the log-likelihood of the data $X_k$. The
term~$p_n$ is the log-prior of hypothesis~$H_n$, which is a way of representing
the relative confidence in hypothesis~$H_n$ prior to seeing the data. In
practice, it can also be interpreted as (a) a complexity term, (b) the
log-posterior resulting from ``previous'' inference steps, or (c) an
initialization term for the inference procedure. The term~$c$ is the
log-probability of the data, which is constant over the hypotheses, and thus
does not affect our analysis. Hence, log-posteriors are compared by their
differences in $l_n + p_n$. Ideally, the log-posterior should be maximum for
the true hypothesis $H_n = H_\ast$. However, since $\omega$ is chosen randomly,
the log-posterior $\log \belief(H_n|X_k)$ is a random quantity. If its variance
is high enough, then a particular realization of the data can lead to a
log-posterior favoring some ``wrong'' hypotheses over the true hypothesis,
i.e.\ $l_n + p_n > l_\ast + p_\ast$ for some $H_n \neq H_\ast$. In general,
this is an unavoidable problem (that necessarily haunts \emph{every}
statistical inference method). Further insight can be gained by analyzing the
expected log-posterior:
\[
    \underbrace{ \sum_{X_k} \belief(X_k|H_\ast) \log \belief(X_k|H_n) }_{L_n}
    \,+\, \underbrace{ \vphantom{\sum_{X_k}} \log \belief(H_n) }_{P_n = p_n}
    \,-\, \underbrace{ \sum_{X_k} \belief(X_k|H_\ast) \log \belief(X_k) }_C.
\]
This reveals\footnote{For $p_i, q_i$ probabilities, $\sum_i p_i \log q_i$ is
maximum when $q_i = p_i$ for fixed $p_i$.} that, \emph{on average}, the
log-likelihood~$L_n$ is indeed maximized by $H_n = H_\ast$. Hence, the
posterior belief will, on average, concentrate its mass on the hypotheses
having high $L_n + P_n$.

\subsection{Conditioning on Events with Zero Belief}

There is one technical point that merits closer inspection. Consider two
events~$A, B \in \fs{F}$ such that $B \cap A \neq \varnothing$ but $\belief(B)
= 0$. One has that
\[
    \truth(A|B) =
    \begin{cases}
        1 & \text{if $B \subset A$} \\
        {?} & \text{else}
    \end{cases}
    \quad \text{and} \quad
    \belief(A \cap B) = \belief(B) \, \belief(A|B) = 0
\]
due to the definition of the truth function and due to Axiom~B4. From this, we
conclude that $\belief(A \cap B) = 0$. For $\belief(A|B)$ there are two
possible cases. If $B \subset A$, then $\belief(A|B) = 1$ due to Axiom~B2.
However, if $B \not \subset A$, then $\belief(A|B)$ is independent of the
degree of belief $\belief(C)$ of any event $C \in \fs{F}$. More generally, if
$D \in \fs{F}$ is such that $\belief(B|D) = 0$, then the value of
$\belief(A|B)$ is independent of the degree of belief $\belief(C|D)$ of any
event $C \in \fs{F}$.

The bottom line is that conditioning on an event with zero belief is a
well-defined operation under the belief axioms outlined in
Definition~\ref{def:belief-axioms}. This is not so in the case of the
probability axioms of measure theory. In measure theory, the probability
measure is a global measure $\mu$ over $\fs{F}$, i.e.\ a function assigning
probability mass $\mu(A)$ to any event~$A \in \fs{F}$. However, implicit in
this definition is the fact that these masses are measured w.r.t. the certain
event $\Omega$. Because of this, the information contained in the probability
measure~$\mu$ is insufficient to uniquely determine the conditional probability
measure $\mu(\cdot|B)$ arising from conditioning on an event $B \in \fs{F}$
having $\mu(B) = 0$. In contrast, the belief function $\belief$ is a
well-defined measure w.r.t. any conditioning event $B \in \fs{F}$, i.e.\
assigning probability mass $\belief(A|B)$ to any event~$A \in \fs{F}$.

\section{Causality}

Suppose there is an unknown cause influencing a result we are waiting for. As
soon as we observe the result, we learn something about the unknown cause.
However, if instead we decide to interrupt the natural regime of the process by
choosing the result ourselves, then our knowledge about the unknown cause will
not change. \emph{This is simply because we know that our current actions
cannot change the past anymore.} Meanwhile, in both cases, we learn something
about the future, i.e.\ about all the outcomes that will follow the result.

This distinction between belief updates following externally generated
observations and internally generated actions is not modeled in Bayesian
probability theory. Essentially, the theory lacks the formal tools to deal with
indeterminate outcomes chosen by the reasoner himself. This requires
introducing additional information to clearly identify the past and the future
of choices, or more abstractly speaking, introducing a causal order of events.

\section{Causal Spaces}

The aim of this section is to introduce causal spaces. Causal spaces contain
enough information to characterize the causal structure of a random process.

Let $\Omega$ be a finite set of \defterm{outcomes}. An \defterm{atom set}
$\fs{A}$ is a partition of $\Omega$, and an \defterm{atom} is a member $A \in
\fs{A}$. Given a set $\fs{E}$ of subsets of $\Omega$, define the
\defterm{algebra generated} by $\fs{E}$, written $\sigma(\fs{E})$, as the
smallest algebra over $\Omega$ containing every member of $\fs{E}$.
Furthermore, define the \defterm{atom set generated} by an algebra $\fs{F}$,
written $\alpha(\fs{F})$, as the largest set of atoms containing members of
$\fs{F}$. For any set $\fs{E}$ of subsets of $\Omega$, we also abbreviate
$\alpha(\fs{E}) \define \alpha(\sigma(\fs{E}))$.

\begin{remark}
In the finite case, it is easily seen that both generated algebras and
generated atom sets are unique.
\end{remark}

\begin{definition}[Primitive Events]
Let $E = (E_0, E_1, E_2, \ldots, E_N)$ be a finite sequence of subsets of
$\Omega$ called \defterm{primitive events}, where $E_0 \define \Omega$, and
where for all $n \geq 1$,
\[
    E_n \notin \sigma\biggl( \{ E_0, E_1, \ldots, E_{n-1} \} \biggr).
\]
Furthermore, define $\fs{E}_n \define \{E_n, E_n^c\}$ and $\fs{A}_0, \fs{A}_1,
\ldots, \fs{A}_N$ as the sequence of atom sets
\[
    \fs{A}_n \define \alpha\biggl( \{ E_0, E_1, \ldots, E_n \} \biggr).
\]
\end{definition}

This setup is illustrated in Figure~\ref{fig:causal-space}. The sequence of
primitive events is an abstract characterization of a random process that
occurs in discrete steps $n=1, 2, \ldots, N$. Each step~$n$ is associated with
a primitive event~$E_n$ representing a basic proposition whose truth value is
resolved during this step (and not before!), i.e.\ step~$n$ determines whether
the outcome $\omega \in \Omega$ is either in~$E_n$ or in~$E_n^c$. The $n$-th
atom set $\fs{A}_n$ contains one proposition for each possible path the random
process can take. Therefore, after~$n$ steps, the process will find itself in
one (and only one) of the members in~$\fs{A}_n$.

\begin{figure}[htbp]
\begin{center}
    \small
    \psfrag{s0}[c]{$E_0$}
    \psfrag{s1}[c]{$E_1$}
    \psfrag{s2}[c]{$E_2$}
    \psfrag{s3}[c]{$E_3$}
    \psfrag{a0}[c]{$\fs{A}_0$}
    \psfrag{a1}[c]{$\fs{A}_1$}
    \psfrag{a2}[c]{$\fs{A}_2$}
    \psfrag{a3}[c]{$\fs{A}_3$}
    \includegraphics[]{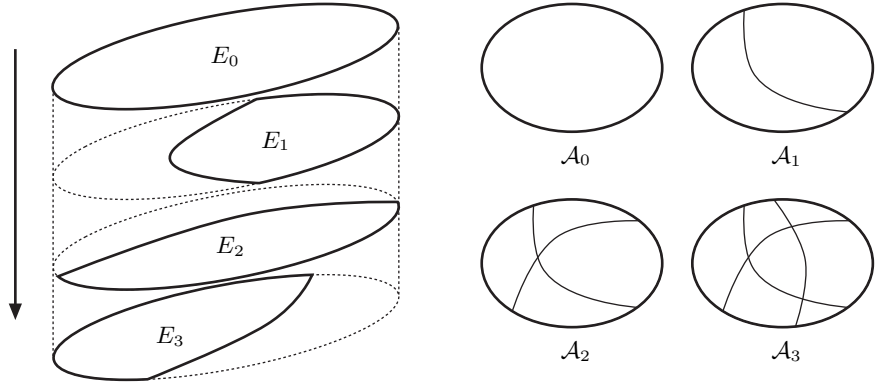}
    \caption{Primitive Events and their Atom Sets.}
    \label{fig:causal-space}
\end{center}
\end{figure}

\begin{remark}
The condition that $E_n$ cannot be in the algebra generated by the previous
events $E_0, \ldots, E_{n-1}$ guarantees that $E_n$ adds a new proposition that
cannot be expressed in terms of the previous propositions.
\end{remark}

The sequence of primitive events $E = (E_1, \ldots, E_N)$ can equivalently be
represented by any sequence $E' = (E'_1, \ldots, E'_N)$ where $E'_n \in
\fs{E}_n$. Due to this, we will call any member of $\fs{E}_n$ primitive event.
We introduce causal functions.

\begin{definition}[Causal Axioms]
Let $\Omega$ be a set of outcomes, and let $E = (E_1, \ldots, E_N)$ be a
sequence of primitive events. A set function $\cause_n$ is a
\defterm{$n$-th causal function} iff
\begin{itemize}
    \item[C1.] $A \in \fs{E}_n, B \in \fs{A}_{n-1},
        \quad \cause_n(A|B) \in [0,1]$.
    \item[C2.] $A \in \fs{E}_n, B \in \fs{A}_{n-1},
        \quad \cause_n(A|B) = 1 \quad$ if $B \subset A$.
    \item[C3.] $A \in \fs{E}_n, B \in \fs{A}_{n-1},
        \quad \cause_n(A|B) = 0 \quad$ if $A \cap B = \varnothing$.
    \item[C4.] $A \in \fs{E}_n, B \in \fs{A}_{n-1},
        \quad \cause_n(A|B) + \cause_n(A^c|B) = 1$.
\end{itemize}
Hence, $\cause_n$ maps $\fs{E}_n \times \fs{A}_{n-1}$ into $[0,1]$. A
\defterm{causal function} over $E$ is a function
\[
    \cause(A|B) = \cause_n(A|B), \qquad \text{if $A \in \fs{E}_n, B \in
    \fs{A}_{n-1}$},
\]
where $\cause_n$ is an $n$-th causal function. Hence, $\cause$ maps $\bigcup_n
(\fs{E}_n \times \fs{A}_{n-1})$ into $[0,1]$.
\end{definition}

The intuition behind this definition is as follows. The causal function
specifies the knowledge the reasoner has about the evolution of a random
process. It specifies the likelihood of a primitive event $A \in \fs{E}_n$ to
happen after the random process is known to have taken a path $B \in
\fs{A}_{n-1}$.

By comparing Axioms~C1--C4 with Axioms~B1--B5
(Section~\ref{sec:reasoning-under-uncertainty}) of belief functions, we observe
the following. First, in contrast to $\belief$, only a subset of combinations
$(A,B) \in \fs{F} \times \fs{F}$ is specified for $\cause$, namely, the ones
that chain a history of primitive events $B \in \fs{A}_{n-1} \subset \fs{F}$
together with the primitive event $A \in \fs{E}_n \subset \fs{F}$ that
immediately follows. Second, Axioms~C1--C4 play the same r\^{o}le as
Axioms~B1--B4, namely: (C1) probabilities lie in the unit interval $[0,1]$;
(C2~\&~C3) probabilities are consistent with the truth function; and (C4)
probabilities of complementary events add up to one. No axiom analogous to
Axiom~B5 is needed for $\cause$.

Putting everything together, one gets a causal space. A causal space contains
enough information to derive an associated belief space.

\begin{definition}[Causal Space]
A \defterm{causal space} is a tuple $(\Omega, E, \cause)$, where: $\Omega$ is a
set of outcomes, $E$ is sequence of primitive events, and $\cause$ is a causal
function over $E$.
\end{definition}

\begin{definition}[Induced Belief Space]
Given a causal space $(\Omega, E, \cause)$, the \defterm{induced belief space}
is the belief space $(\Omega, \fs{F}, \belief)$ where the algebra $\fs{F}$ and
the belief function $\belief$ are defined as
\begin{itemize}
    \item[i.]  $\fs{F} = \sigma\Bigl( \{E_0, E_1, \ldots, E_N\} \Bigr)$;
    \item[ii.] $\belief(A|B) = \cause(A|B)$,
        \quad for all $(A,B) \in \bigcup_n (\fs{E}_n \times \fs{A}_{n-1})$.
\end{itemize}
\end{definition}

Thus, the induced belief space is constructed by generating the algebra
$\fs{F}$ from the primitive events $E$, and by equating the belief function
$\belief$ to the causal function $\cause$ over the subset of $\fs{F} \times
\fs{F}$ where $\cause$ is defined. The following theorem tells us that this
subset is enough to completely determine the whole belief function.

\begin{theorem}
The induced belief space exists and is unique.
\end{theorem}
\begin{proof}
Let $\fs{F}_0, \fs{F}_1, \ldots, \fs{F}_N$ denote the sequence of algebras
generated as
\[
    \fs{F}_n \define \sigma( \{ E_0, E_1, \ldots, E_n \} ).
\]
Let $r, s \in \nats$, $r \leq s$, be the smallest numbers such that $B$ is
$\fs{F}_r$-measurable and $A$ is $\fs{F}_s$-measurable. Let $\fs{B} \subset
\fs{A}_r$ and $\fs{A} \subset \fs{A}_s$ be the partitions of $B$ and $A$
respectively. Then, $\belief(A|B) = 0$ if $A \cap B$ by the belief axioms, and
\[
    \belief(A|B) = \sum_{a \in \fs{A}} \belief(a|B)
\]
otherwise, because the members $a$ of $\fs{A}$ are disjoint. For every $a \in
\fs{A}$, let $b \in \fs{B}$ be the unique member of the partition of $B$ such
that $a \subset b$. Obviously,
\[
    \belief(a|B) = \belief(a|b),
\]
because $a \cap B = a \cap b$. Let $a^1, a^2, \ldots, a^s$ the unique sequence
$a^j \in \fs{E}_j$ such that
\[
    a
    = a^1 \cap a^2 \cap \cdots \cap a^s
    = \bigcap_{j=1}^s a^j = b \cap \bigcap_{j=r+1}^s a^j,
\]
where the last equality comes from $b = a^1 \cap \cdots \cap a^r$. Hence,
\[
    \belief(a|b)
    = \belief\Bigl(\bigcap_{j=r+1}^s a^j \Bigr| b \Bigr)
    = \prod_{j=r+1}^s \belief\Bigl(a^j \Bigr| b \cap \bigcap_{i=r+1}^{j-1} a^i \Bigr)
    = \prod_{j=r+1}^s \cause\Bigl(a^j \Bigr| b \cap \bigcap_{i=r+1}^{j-1} a^i
    \Bigr).
\]
The last replacement can be done because $a^j \in \fs{E}_j$ and $b \cap
\bigcap_{i=r+1}^{j-1} a^i \in \fs{A}_{j-1}$. Thus, we have proven the
following. First, $\fs{F}$ is unique because generated algebras are unique.
Second, we have shown, for arbitrarily chosen events $A, B \in \fs{F}$, how to
reexpress $\belief(A|B)$ into an expression involving only terms of the form
$\cause(C|D)$. Hence, it cannot be that $\belief, \belief'$ are both consistent
with $\cause$ and there is $A, B \in \fs{F}$ such that $\belief(A|B) \neq
\belief'(A|B)$.
\end{proof}

We now define the operation that specifies how the knowledge about the random
process transforms when the reasoner himself intervenes it.

\begin{definition}[Intervention]\label{def:intervention}
Given a causal space $(\Omega, E, \cause)$ and a primitive event $A \in
\fs{E}_n$ for some $n \in \{1, \ldots, N\}$, the $A$-\defterm{intervention} is
the causal space $(\Omega, E, \cause')$ where for all $(B,C) \in \bigcup_n
(\fs{E}_n \times \fs{A}_{n-1})$,
\[
    \cause'(B|C) =
    \begin{cases}
        1 & \text{if $A = B\phantom{^c}$ and $(B \cap C) \notin \{\varnothing, C\}$,} \\
        0 & \text{if $A = B^c$ and $(B \cap C) \notin \{\varnothing, C\}$,} \\
        \cause(B|C) & \text{else.} \\
    \end{cases}
\]
\end{definition}

This is an important definition. The reasoner ask himself the question: ``How
do my beliefs about the world change if I were to choose the truth value of a
primitive event?'' This is answered by \emph{directly changing the causal
function accordingly} (Figure~\ref{fig:intervention}). However, this change
cannot contradict the logical constraints given by the underlying truth
function.

\begin{remark}
Note that $(B \cap C) \notin \{\varnothing, C\} \Leftrightarrow \truth(B|C) =
{?}$. Hence, an intervention can only affect primitive propositions $B \in
\fs{E}_n$ that have an unresolved truth value given the history $C \in
\fs{A}_{n-1}$. Moreover, the intervention resolves the truth value of $B$. This
makes intuitively sense.
\end{remark}

\begin{figure}[htbp]
\begin{center}
    \small
    \psfrag{l1}[c]{(a)}
    \psfrag{l2}[c]{(b)}
    \psfrag{e0}[c]{$E_0$}
    \psfrag{e1}[c]{$E_1$}
    \psfrag{e2}[c]{$E_2$}
    \psfrag{e3}[c]{$E_3$}
    \includegraphics[]{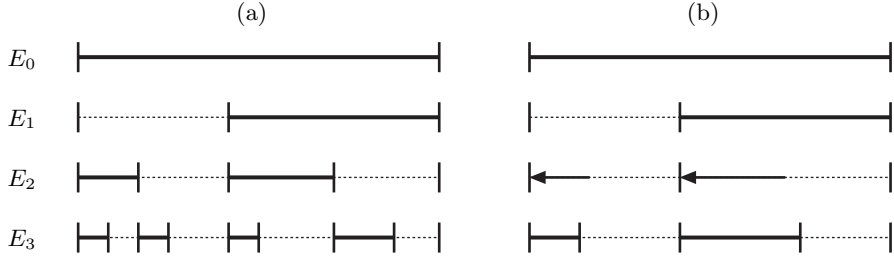}
    \caption[Intervention.]{An Intervention. The primitive events
    $E = (E_0, E_1, E_2, E_3)$ are sets on the unit interval. Panels~(a)
    and~(b) show a the causal space before and after an
    $E_2^c$-intervention respectively. This representation
    shows the atom sets $\fs{A}_0$ to $\fs{A}_3$ and conditional
    probabilities (given by the relative lengths).}
    \label{fig:intervention}
\end{center}
\end{figure}

We will use the abbreviation $\hat{A}$ to denote $A$-interventions on a causal
space. When the underlying causal space $(\Omega, E, \cause)$ inducing a belief
space $(\Omega, \fs{F}, \belief)$ is clear from the context, then the
expression $\belief(B|\hat{A})$ denotes the belief $\belief'(B|A)$ measured
w.r.t. the belief space $(\Omega, \fs{F}, \belief')$ induced by the
$A$-intervention of $(\Omega, E, \cause)$. Furthermore, when $A \in \fs{F}$ is
an event such that
\[
    A = \bigcap_{i=1}^I A_i,
\]
where each $A_i$ is a primitive event, then the $A$-intervention is the causal
space resulting as the succession of $A_i$-interventions.

\section{Concluding Remarks}

We have shown how to derive a simple framework for reasoning under uncertainty
and intervention. This is achieved in three steps. First, we have restated
logic in set-theoretic terms to obtain a framework for reasoning under
certainty. Second, we have extended this framework to model reasoning under
uncertainty. Finally, we have introduced causal spaces and shown how it
provides enough information to model knowledge containing causal information
about the world.

This framework can be extended in many ways. Importantly, it has been designed
to be consistent with the literature on Bayesian statistics \citep{Cox1961,
Jaynes2003} and the literature on causality based on graphs \citep{Pearl2000,
Spirtes2000, Dawid2010} and probability trees \citep{Shafer1996}.

\vskip 0.2in
\bibliographystyle{plainnat}
\small
\bibliography{bibliography}

\end{document}